\documentclass[reqno]{amsart}

\usepackage[latin1]{inputenc}

\usepackage{enumerate}
\usepackage{amstext,amsmath,amsthm,amsfonts,amssymb,amscd}
\usepackage{latexsym,mathrsfs,eufrak,dsfont,euscript}

\usepackage{color}
\definecolor{rojo}{RGB}{221,0,0}

\usepackage{hyperref} 

\usepackage{pstricks}
\usepackage{pst-plot}
\usepackage{pst-all}
\usepackage{graphicx}

\newtheorem{theorem}{Theorem}
\newtheorem{proposition}[theorem]{Proposition}
\newtheorem{lemma}[theorem]{Lemma}

\theoremstyle{definition}

\theoremstyle{remark}

\numberwithin{equation}{section}

\newcommand{\abs}[1]{\left\vert#1\right\vert}
\newcommand{\set}[1]{\left\{#1\right\}}

\newcommand{\norm}[1]{\left\Vert#1\right\Vert}

\def\zR{\ensuremath{\mathbb{R}}}

\def\pii{\left(}
\def\pdd{\right)}
\def\cii{\left[}
\def\cdd{\right]}
\def\lii{\left\{}
\def\ldd{\right\}}

\begin{document}
\title []{Mean value formulas for solutions of some degenerate elliptic equations and applications}
%
%


\author[]{Hugo Aimar}
\email{haimar@santafe-conicet.gov.ar}
\author[]{Gast\'{o}n Beltritti}
\email{gbeltritti@santafe-conicet.gov.ar}
\author[]{Ivana G\'{o}mez}
\email{ivanagomez@santafe-conicet.gov.ar}
\thanks{The research was supported  by CONICET, ANPCyT (MINCyT) and UNL}

\begin{abstract}
We prove a mean value formula for weak solutions of $\hbox{div}(|y|^{a}\hbox{grad}\\u)=0$ in $\mathbb{R}^{n+1}=\{ (x,y):\ x\in \mathbb{R}^{n},\ y\in\mathbb{R} \}$, $-1<a<1$ and balls centered at points of the form $(x,0)$. We obtain an explicit nonlocal kernel for the mean value formula for solutions of $(-\triangle)^{s}f=0$ on a domain $D$ of $\mathbb{R}^{n}$. When $D$ is Lipschitz we prove a Besov type regularity improvement for the solutions of  $(-\triangle)^{s}f=0$.
\end{abstract}

\maketitle

\section*{Introduction}
In \cite{CaSi2007}, L.~Caffarelli and L.~Silvestre show how the fractional powers of $-\triangle$ in $\zR^{n}$ can be obtained as Dirichlet to Neumann type operators in the extended domain $\zR^{n+1}$. The operator in the extended domain is given by $L_{a}u=\hbox{div}\pii |y|^{a} \text{\hbox{grad}}\ u \pdd$, where $a\in (-1,1)$, $u=u(x,y)$, $x\in\zR^{n}$, $y\in\zR$ and \textit{div} and \textit{grad} are the standard divergence and gradient operators in $\zR^{n+1}=\lii (x,y) : x\in\zR^{n},\ y\in\zR  \ldd$. The exponent $a$ is related to the fractional power of the Laplacian $(-\triangle)^{s}$ through $2s=1-a$. Notice that when $a=0$ the operator $L_{a}$ is the Laplacian in $\zR^{n+1}$ and $s=\frac{1}{2}$. The theory of H\"{o}lder regularity of solutions through Harnack's inequalities, is one of the several results in \cite{CaSi2007}. This theory has been extended in \cite{ToSti2010} to other second order partial differential operators including the harmonic oscillator.

\medskip
Since for $a\in (-1,1)$ the weight $w(x,y)=|y|^{a}$ belongs to the Muckenhoupt class $A_{2}(\zR^{n+1})$, the regularity theory developed by Fabes, Kenig and Serapioni in \cite{FaKeSe82}, can be applied. The fact that $w$ is in $A_{2}(\zR^{n+1})$ follows easily from the fact that it is a product of the weight which is constant and equal to one in $\zR^{n}$ times the $A_{2}(\zR)$ weight $|y|^{a}$ for $a\in (-1,1)$. In particular Harnack's inequality and H\"{o}lder regularity of solutions are available.

\medskip
It seems to be clear that, when $a\neq 0$, the weight $w(x,y)=|y|^{a}$ introduces a bias  which prevents us from expecting mean values on spherical objects in $\zR^{n+1}$. Except at $y=0$, where the symmetry of $w$  with respect to the hyperplane $y=0$ may bring back to spheres their classical role. In \cite{FaGa87} some generalizations of classical mean value formulas are also considered.

\medskip
By choosing adequate test functions we shall prove the mean value formula, for balls centered at the hyperplane $y=0$, for weak solutions $v$ of $L_a v=0$.

\medskip
The above considerations would only allow mean values  for solutions with balls centered at such small sets as the hyperplane $y=0$ of $\zR^{n+1}$. But it turns out that this suffice to get mean value formulas for solutions of $(-\triangle)^{s}f=0$.

\medskip
In \cite{Silvestretesis} a mean value formula is proved as Proposition~2.2.13,  see also \cite{Landkofbook}.  In order obtain improvement results for the Besov regularity of solutions of $(-\triangle)^{s}f=0$ in the spirit of \cite{DaDeV97}  and \cite{AGconstructive}, our formula seems to be more suitable because we can get explicit estimates for the gradients of the mean value kernel. Regarding Besov regularity of harmonic functions see also \cite{JeKe95}.

\medskip
The paper is organized in three sections. In the first one we prove mean value formulas for solutions of $L_{a}u=0$ at the points on the hyperplane $y=0$ of $\zR^{n+1}$. The second section is devoted to apply the result in Section~\ref{sec:meanvaluedivergence} in order to obtain a nonlocal mean value formula for solutions of $(-\triangle)^{s}f=0$ on domains of $\zR^{n}$.
Finally, in Section~\ref{sec:gradientestimates}, we use the above results to obtain a Besov regularity improvement for solutions of $(-\triangle)^{s}f=0$ in Lipschitz domains of $\zR^{n}$. At this point we would like to mention the recent results in \cite{NOS} in relation with the rate of convergence of nonlinear approximation methods observed by Dahlke and DeVore in the harmonic case.

\section{Mean value formula for solutions of $L_{a}u=0$}\label{sec:meanvaluedivergence}

Let $D$ be a domain in $\zR^{n}$. Let $\Omega$ be the open set in $\zR^{n+1}$ given by $\Omega=D\times(-d,d)$ with $d$ the diameter of $D$. Notice that for $x\in D$ and $r>0$ such that $B(x,r)\subset D$, then $S((x,0),r)\subset\Omega$ where $B$ denotes balls in $\zR^{n}$ and $S$ denotes the balls in $\zR^{n+1}$. With $H^1(\abs{y}^{a})$ we denote the Sobolev space of those functions in $L^2(\abs{y}^a dx dy)$ for which $\nabla f$ belongs to $L^2(\abs{y}^a dx dy)$.

A weak solution $v$ of $L_{a}v=0$ in $\Omega$ is a function in the weighted Sobolev space $H^{1}(|y|^{a})$, such that

\begin{displaymath}
\iint_{\Omega}\nabla v\cdot \nabla \psi |y|^{a} \ dx dy=0
\end{displaymath}
for every test function $\psi$ supported in $\Omega$.

The main result of this section is contained in the next statement. As in \cite{CaSi2007} we shall use $X$ to denote the points $(x,y)$ in $\zR^{n+1}$ with $x\in\zR^{n}$ and $y\in\zR$. For $x\in D$ with $\delta(x)$ we shall denote the distance from $x$ to $\partial D$.

\begin{theorem}\label{teovmv}
Let $v$ be a weak solution of $L_{a}v=0$ in $\Omega$. Let $\varphi(X)=\eta(\abs{X})$, $\eta\in C_{0}^{\infty}(\zR^{+})$ supported in the interval $\cii\frac{1}{4},\frac{3}{4}\cdd$ and $\iint_{\zR^{n+1}}\varphi(X)|y|^{a}\ d X=1$ be given. If $x\in D$ and $0<r<\delta(x)$, then
\begin{displaymath}
v(x,0)=\iint_{\Omega}\varphi_{r}(x-z,-y)v(z,y)|y|^{a} dz dy
\end{displaymath}
with
\begin{displaymath}
\varphi_{r}(X)=\frac{1}{r^{n+1+a}}\varphi\pii \frac{X}{r} \pdd.
\end{displaymath}
\end{theorem}

\begin{proof}
Set $A=\int_{0}^{\infty}\rho \eta(\rho) d\rho$ and $\zeta(t)=\int_{0}^{t}\rho\eta(\rho) d\rho-A$. Notice that $\zeta(t)\equiv 0$ for $t\geq \frac{3}{4}$ and $\zeta(t)\equiv-A$ for $0\leq t\leq \frac{1}{4}$. The function $\psi(X)=\zeta(\abs{X})$ is, then, in $C^{\infty}(\zR^{n+1})$ and has compact support in the ball $S((0,0),1)$. It is easy to check that $\nabla\psi(X)=\varphi(X)X$. Take now $x\in D$ and $0<r<\delta(x)$. Set $\varphi_{r}(Z)=r^{-n-1-a}\varphi(r^{-1}Z)$, $Z\in\zR^{n+1}$, and define

\begin{displaymath}
\Phi_{x}(r)=\iint_{\Omega}\varphi_{r}(X-Z)v(Z)|y|^{a} dZ,
\end{displaymath}
where $X=(x,0)$, $Z=(z,y)$, $dZ=dzdy$ and $v$ is a weak solution of $L_{a}v=0$ is $\Omega$. As usual, we aim to prove that $\Phi_{x}(r)$ is a constant function of $r$ and that $\lim\limits_{r\rightarrow 0}\Phi_{x}(r)=v(X)$. From the results in \cite{FaKeSe82} with $w(Z)=|y|^{a}$, which belongs to the Muckenhoupt class $A_{2}(\zR^{n+1})$ when $-1<a<1$, we know that $v$ is H\"{o}lder continuous on each compact subset of $\Omega$. Then the convergence $\Phi_{x}(r)\rightarrow v(X)=v(x,0)$ as $r\rightarrow 0$, follows from the fact that
\begin{displaymath}
\iint\varphi_{r}(Z)|y|^{a} dZ=\frac{1}{r^{a+1+n}}\iint\varphi\pii\frac{z}{r},\frac{y}{r}\pdd|y|^{a} dz dy =1.
\end{displaymath}

In order to prove that $\Phi_{x}(r)$ is constant as a function of $r$ we shall take its derivative with respect to $r$ for fixed $x$. Notice first that
\begin{displaymath}
\Phi_{x}(r)=\iint_{S((0,0),1)}\varphi(Z)v(X-rZ)|y|^{a} dz dy.
\end{displaymath}
Since $\nabla v\in L^{2}(|y|^{a} dX)$ we have

\begin{displaymath}
\begin{split}
\frac{d}{dr}\Phi_{x}(r)&=-\iint_{S((0,0),1)}\varphi(Z)\nabla v(X-rZ)\cdot Z|y|^{a} dZ\\
&=-\iint_{S((0,0),1)}\nabla v(X-rZ)\cdot\nabla\psi(Z)|y|^{a} dZ\\
&=-\frac{1}{r^{a+1+n}}\iint_{\Omega}\nabla v(Z)\cdot\nabla\psi\pii\frac{X-Z}{r} \pdd|y|^{a} dZ\\
&=\iint_{\Omega}\nabla v(Z)\cdot\nabla\cii \frac{1}{r^{n+a}}\psi\pii\frac{X-Z}{r} \pdd \cdd|y|^{a} dZ,
\end{split}
\end{displaymath}
which vanishes since $\frac{1}{r^{n+a}}\psi\pii \frac{X-Z}{r} \pdd$ as a function of $Z$ is a test function for the fact that $v$ solves $L_{a}v=0$ in $\Omega$.
\end{proof}

\section{Mean value formula for solutions of $(-\triangle)^{s}f=0$}\label{sec:meanvaluelaplacian}

In this section we shall use the results and we shall closely follow the notation in \cite{CaSi2007}. Take $f\in L^{1}\bigl(\zR^{n},\frac{dx}{(1+|x|)^{n+2s}} \bigr)$ with $(-\triangle)^{s}f=0$ on the domain $D\in\zR^{n}$. Then, with $u(x,y)=\pii P_{y}^{a}\ast f \pdd(x)$ and $P_{y}^{a}(x)=C y^{1-a}\pii |x|^{2} + y^{2} \pdd^{-\tfrac{n+1-a}{2}}$ the function
\begin{displaymath}
v(x,y)=\left\{ \begin{array}{cc}
 u(x,y) &    \text{in}\ D\times \zR^{+}  \\
u(x,-y) & \text{in}\ D\times \zR^{-}
\end{array}\right.
\end{displaymath}
is a weak solution of $L_{a}v=0$ in $D\times\zR$. In particular $v$ is H\"{o}lder continuous in $D\times\zR$ from the results in \cite{FaKeSe82}. Theorem~\ref{teovmv} guarantees that, for $0<r<\delta(x)$ and $x\in D$,

\begin{equation}\label{form.mv.f}
f(x)=u(x,0)=v(x,0)=\iint\varphi_{r}(X-Z)v(Z)|y|^{a} dZ
\end{equation}
where, as before, $X=(x,0)$ and $Z=(z,y)$. On the other hand, the definitions of $v$ and $u$ provide the formula
\begin{equation}\label{form.mv.f1}
v(Z)=v(z,y)=\pii P_{\abs{y}}^{a}\ast f \pdd(z).
\end{equation}

Replacing \eqref{form.mv.f1} in \eqref{form.mv.f}, provided that the interchange of the order of integration holds, we obtain the main result of this section.

\begin{theorem}\label{teopri}
Let $0<s<1$ be given. Assume that $D$ in an open set in $\zR^{n}$
on which $(-\triangle)^{s}f=0$. Then for every $x\in D$ and every
$0<r<\delta(x)$ we have that $f(x)=\pii\Phi_{r}\ast f
\pdd(x)$, where $\Phi_{r}(x)=r^{-n}\Phi\pii\frac{x}{r}
\pdd$,
$\Phi(x)=\int_{y\in\zR}\int_{z\in\zR^{n}}\varphi(z,-y)P^{a}_{|y|}(x-z)|y|^{a}\
dz dy$, $\varphi_{r}(x,y)=r^{-(n+1+a)}\varphi\pii
\frac{x}{r},\frac{y}{r} \pdd$, $\varphi$ is a
$C^{\infty}(\zR^{n+1})$ radial function supported in the unit ball
of $\zR^{n+1}$ with $\iint_{\zR^{n+1}}\varphi(x,y)|y|^{a}\
dx dy=1$ and $P^{a}_{y}$ is a constant times
$y^{1-a}\pii |x|^{2}+y^{2}\pdd^{-\frac{n+1-a}{2}}$.
\end{theorem}

\begin{proof}
Inserting \eqref{form.mv.f1} in \eqref{form.mv.f} we have
\begin{displaymath}
\begin{split}
f(x)&=v(x,0)=\iint \varphi_{r}(x-z,-y)v(z,y)\ |y|^{a}dz dy\\
&=\iint\varphi_{r}(x-z,y)(P^{a}_{|y|}\ast f)(z)|y|^{a} dz dy\\
&=\int_{y\in\zR}\int_{z\in\zR^{n}}\varphi_{r}(x-z,-y)\pii
\int_{\bar{z}\in
\zR^{n}} P^{a}_{|y|}(z-\bar{z})f(\bar{z})\ d\bar{z} \pdd|y|^{a} dz  dy\\
&=\int_{\bar{z}\in\zR^{n}}\pii
\int_{y\in\zR}\int_{z\in\zR^{n}} \varphi_{r}(x-z,-y) P^{a}_{|y|}(z-\bar{z})
|y|^{a} dz  dy \pdd f(\bar{z}) d\bar{z}\\
&=\int_{\bar{z}\in\zR^{n}}\Phi_{r}(x,\bar{z})f(\bar{z}) d\bar{z},
\end{split}
\end{displaymath}
with $\Phi_r(x,\bar{z})=\int_{y\in \mathbb{R}}\int_{z\in \mathbb{R}^n}\varphi_r(x-z,-y)P^a_{\abs{y}}(z-\bar{z})\abs{y}^a dz dy$.
The last equality in the above formula follows from the fact that $\tfrac{f(\bar{z})}{(1+\abs{\bar{z}}^2)^{\tfrac{n+1-a}{2}}}$
is integrable in $\mathbb{R}^n$, since
\begin{equation*}
\int_{y\in\zR}\int_{z\in\zR^{n}}|\varphi(x-z,-y)|P^{a}_{|y|}(z-\bar{z})|y|^{a} dz dy\leq \frac{C}{(1+\abs{\bar{z}}^2)^{\tfrac{n+1-a}{2}}}
\end{equation*}
for some positive constant $C$. In fact, on one hand
\begin{align}\label{eq:estimatePhi}
\int_{y\in\zR}\int_{z\in\zR^{n}} &|\varphi(x-z,-y)| P^{a}_{|y|}(z-\bar{z})|y|^{a} dz dy\notag\\
& \leq \int_{-1}^{1} \norm{\varphi(x-\cdot,y)}_{L^{\infty}}\norm{P^{a}_{|y|}(\cdot-\bar{z})}_{L^1}\abs{y}^a dy\leq C;
\end{align}
on the other, for $\abs{\bar{z}-x}>2$ we have
\begin{align}\label{eq:estimatePhiawayorigin}
\int_{y\in\zR}\int_{z\in\zR^{n}}|& \varphi(x-z,-y)| P^{a}_{|y|}(z-\bar{z})|y|^{a} dz dy \notag\\
& \leq C\iint_{S((x,0),1)}\frac{\abs{y}}{(y^2+\abs{z-\bar{z}}^2)^{\tfrac{n+1-a}{2}}} dz dy \notag\\
& \leq \frac{C}{\abs{x-\bar{z}}^{n+1-a}}.
\end{align}
So that $\Phi_r(x,\bar{z})\leq \frac{C(r)}{(1+\abs{x-\bar{z}})^{n+1-a}}\leq \frac{C(x,r)}{(1+\abs{x})^{n+1-a}}$, hence $\int\Phi_{r}(x,\bar{z})f(\bar{z}) d\bar{z}$ is absolutely convergent.
It remains to prove that $\Phi_r(x,\bar{z})=\tfrac{1}{r^n}\Phi(\tfrac{x-\bar{z}}{r})$ with
$\Phi(x)=\int_{y\in\zR}\int_{z\in\zR^{n}}\varphi(z,-y)P^{a}_{|y|}(x-z)|y|^{a} dz dy$. Let us compute $\Phi(\tfrac{x-\bar{z}}{r})$
changing variables. First in $\mathbb{R}^n$ with $\nu=x-rz$, then in $\mathbb{R}$ with $t=ry$,
\begin{align*}
\Phi\left(\frac{x-\bar{z}}{r}\right) &= \int_{y\in\zR}\int_{z\in\zR^{n}}\varphi(z,-y)P^{a}_{|y|}\Bigl(\frac{x-\bar{z}-rz}{r}\Bigr)|y|^{a} dz dy\\
&=\int_{y\in\zR}\int_{\nu\in\zR^{n}}\frac{1}{r^n}\varphi\Bigl(\frac{x-\nu}{r},-y\Bigr)P^{a}_{|y|}\Bigl(\frac{\nu-\bar{z}}{r}\Bigr)|y|^{a} d\nu dy\\
&=\int_{t\in\zR}\int_{\nu\in\zR^{n}}\frac{1}{r^{n+1+a}}\varphi\Bigl(\frac{x-\nu}{r},-\frac{t}{r}\Bigr)P^{a}_{\abs{\tfrac{t}{r}}}\Bigl(\frac{\nu-\bar{z}}{r}\Bigr)|t|^{a} d\nu dt\\
&= r^n\int_{t\in\zR}\int_{\nu\in\zR^{n}}\varphi_r(x-\nu,-t)P^{a}_{\abs{t}}(\nu-\bar{z})|t|^{a} d\nu dt\\
&= r^n\Phi_r(x,\bar{z}),
\end{align*}
as desired.
\end{proof}

We collect in the next result some basic properties of the mean value kernel $\Phi$.

\begin{proposition}\label{propo:propertiesPhi}
The function $\Phi$ defined in the statement of Theorem~\ref{teopri} satisfies the following properties.
\begin{enumerate}[(a)]
\item $\Phi(x)$ is radial;
\item $(1+\abs{x})^{n+1-a}\abs{\Phi(x)}$ is bounded;
\item $\int_{\mathbb{R}^n}\Phi(x) dx=1$;
\item $\sup_{r>0}\abs{(\Phi_r\ast f)(x)}\leq c Mf(x)$, where $M$ is the Hardy-Littlewood maximal opera\-tor in
$\mathbb{R}^n$;
\item if $\Psi^{i}(x)=\frac{\partial \Phi}{\partial x_{i}}(x)$,
then $\Psi^{i}(0)=0$ and $\int\Psi^{i}(x)\ dx=0$;
\item for some constant $C>0$, $\abs{\Psi^{i}(x)}\leq
\frac{C}{\abs{x}^{n+2-a}}$ for $\abs{x}>2$;
\item $\abs{\nabla \Psi^{i}}$ is bounded on $\zR^{n}$ for every
$i=1,\ldots,n$.
\end{enumerate}
\end{proposition}

\begin{proof}
Let $\rho$ be a rotation of $\mathbb{R}^n$, then
\begin{displaymath}
\begin{split}
\Phi(\rho
x)&=\int_{y\in\zR}\int_{z\in\zR^{n}}\varphi(z,-y)P^{a}_{|y|}(\rho
x-z)|y|^{a} dz  dy\\
&=\int_{y\in\zR}\int_{z\in\zR^{n}}\varphi(\rho^{-1}z,-y)P^{a}_{|y|}(\rho^{-1}(\rho
x-z))|y|^{a} dz dy\\
&=\int_{y\in\zR}\int_{z\in\zR^{n}}\varphi(\rho^{-1}z,-y)
P^{a}_{|y|}(x-\rho^{-1}z)|y|^{a} dz dy\\
&=\int_{y\in\zR}\int_{\bar{z}\in\zR^{n}}\varphi(\bar{z},-y)
P^{a}_{|y|}(x-\bar{z})|y|^{a} d\bar{z} dy\\
&=\Phi(x),
\end{split}
\end{displaymath}
which proves \textit{(a)}. Part \textit{(b)} has already been proved in \eqref{eq:estimatePhi} and \eqref{eq:estimatePhiawayorigin}. By taking $f\equiv 1$ in Theorem~\ref{teopri} we get \textit{(c)}.
From \textit{(a)} and \textit{(c)} the estimate of the maximal operator is a classical result (see \cite{Stein70}).
Item \textit{(e)} follows from the fact that $\Phi$ is radial and
smooth and from \textit{(c)}.

Let us now show that $|\Psi^{i}(x)|\leq \frac{C}{|x|^{n+2-a}}$ for $|x|>2$. In fact,

\begin{align}\label{acot.origen.com}
|\Psi^{i}(x)|&= 2\abs{\int_{0}^{\infty}\int_{z\in\zR^{n}}\frac{\partial\varphi}{\partial x_{i}}(z,y) P_{y}^{a}(x-z)y^{a} dz dy}\notag\\
&= 2\abs{\int_{0}^{1}\int_{z\in B(0,1)}\varphi\pii z,y \pdd \frac{\partial}{\partial x_{i}}\pii P_{y}^{a}(x-z)y^{a}\pdd dz dy}\notag\\
&\leq C\int_{0}^{1}\int_{z\in B(0,1)}|\varphi(z,y)|\frac{1}{|x-z|^{n+2-a}} dz dy\notag\\
&\leq \frac{C}{(|x|-1)^{n+2-a}}\int_{0}^{1}\int_{z\in B(0,1)}|\varphi(z,y)| dz dy\notag\\
&\leq \frac{C}{|x|^{n+2-a}}.
\end{align}
By taking the derivatives of the function $\varphi$ the proof of \textit{(g)} proceeds as in \eqref{eq:estimatePhi}.
\end{proof}

\section{Maximal estimates for gradients of solutions of $(-\triangle)^{s}f=0$ in open domains and the improvement of Besov regularity}\label{sec:gradientestimates}

The mean value formula proved in Section~\ref{sec:meanvaluelaplacian} for solutions of
$(-\triangle)^{s}f=0$ in an open domain $D$ of $\zR^{n}$ can be used to obtain improvement of Besov regularity of $f$.
Here we illustrate how  Theorem~\ref{teopri} can be used to get a result in the lines introduced by Dahlke
and DeVore for harmonic functions. We shall prove the following result.

\begin{theorem}\label{teo:improvementBesov}
Let $D$ be a bounded Lipschitz domain in $\zR^{n}$. Let $0<s<1$. Let $1<p<\infty$ and
$0<\lambda<\frac{n-1}{n}$ be given. Assume that $f\in
B_{p}^{\lambda}(\zR^{n})$ and that $(-\triangle)^{s}f=0$ on $D$,
then $f\in B_{\tau}^{\alpha}(D)$ with
$\frac{1}{\tau}=\frac{1}{p}+\frac{\alpha}{n}$ and
$0<\alpha<\lambda\frac{n}{n-1}$.
\end{theorem}

Here $B_{p}^{\lambda}(\zR^{n})$ and $B_{\tau}^{\alpha}(D)$ denote the standard Besov spaces on $\zR^{n}$ and on $D$ with $p=q$ for the usual notation $B_{p,q}^{\lambda}$ of this scale. Among the several descriptions of these spaces the best suited for our purposes is the characterization through wavelet coefficients \cite{Meyer92waveletsI}.

It is worthy noticing that in contrast with the local cases
associated to the harmonic functions in \cite{DaDeV97} and the
temperatures in \cite{AGconstructive}, now the $B_{p}^{\lambda}$
regularity is required on the whole space $\zR^{n}$ and that the
improvement is only in $D$.

The basic scheme is that in \cite{DaDeV97}, and the central tool is then the estimate contained in the next statement.

\begin{lemma}\label{lemma:normgradientsestimates}
Let $D$ be a domain of $\zR^{n}$.  Let $0<\lambda<1$ and
$1<p<\infty$. For $f\in B_{p}^{\lambda}(\zR^{n})$ with
$(-\triangle)^{s}f=0$ on $D$, we have
\begin{displaymath}
\pii\int_{D}\abs{\delta(x)^{1-\lambda} \nabla f(x)}^{p} dx\pdd^{\frac{1}{p}}\leq C\norm{f}_{B_{p}^{\lambda}(\zR^{n})}
\end{displaymath}
where $\delta(x)$ is the distance from $x$ to the boundary of $D$,
$\nabla f$ is the gradient of $f$ and $C$ is a constant.
\end{lemma}

The main difference between the local case in
\cite{DaDeV97} and our nonlocal setting is precisely provided by
the fact that since our mean value kernel is not localized in $D$,
the Calderón maximal operator needs to be taken on the whole
$\zR^{n}$, not only on $D$.

The result is itself a consequence of a pointwise estimate
of the gradient of $f$ in terms of the sharp Calderón maximal
operator and \cite{DeSa84}. The result is contained in the next
statement and follows from the mean value formula in Theorem~\ref{teopri}, and the basic properties of the mean value kernel $\Phi_{r}$ and its first order partial derivatives contained in Proposition~\ref{propo:propertiesPhi}.
.

\begin{lemma}\label{lema.acot.max}
Let $D$ and $\lambda$ be as in Lemma~\ref{lemma:normgradientsestimates} and let $x\in D$ and $0<r<\delta(x)$. Then
\begin{displaymath}
|\nabla f(x)|\leq C r^{\lambda-1}M^{\sharp,\lambda}f(x),
\end{displaymath}
with
\begin{displaymath}
M^{\sharp,\lambda}f(x)=\sup\frac{1}{|B|^{1+\frac{\lambda}{n}}}\int_{B}\abs{f(y)-f(x)}dy
\end{displaymath}
where the supremum is taken on the family of all balls of
$\zR^{n}$ containing $x$.
\end{lemma}

\begin{proof}
 From the definition of $\Phi$ it is clear that
$\frac{\partial}{\partial x_{i}}\Phi_{r}(x)=\frac{1}{r}\Psi_{r}^{i}(x)$ with
$\Psi^{i}(x)=2\int_{0}^{\infty}\int_{z\in\zR^{n}}\frac{\partial
\varphi}{\partial z_{i}}(z,y)P_{y}^{a}(x-z)y^{a} dz dy$,
$i=1,\ldots,n$. Since from \textit{(e)} in Proposition~\ref{propo:propertiesPhi} we have that
$\Psi^{i}(0)=0$, then
\begin{equation}\label{acot.origen}
|\Psi_{r}^{i}(x)|=\abs{\Psi_{r}^{i}(x)-\Psi_{r}^{i}(0)}\leq
|x|\sup\limits_{\xi\in\zR^{n}}|\nabla\Psi_{r}^{i}(\xi)|\leq
\frac{C}{r^{n+1}}|x|,
\end{equation}
from \textit{(g)} in Proposition~\ref{propo:propertiesPhi}. This is a good estimate in a
neighborhood of $0$.
Applying the mean value formula for $f$ we get the
result after the following estimates,
\begin{displaymath}
\begin{split}
\abs{ \frac{\partial f(x)}{\partial x_{i}}}&=\abs{ \frac{\partial }{\partial x_{i}}\pii\Phi_{r}\ast f\pdd(x)}\\
&=\abs{\frac{1}{r}\int_{\zR^{n}}f(x-z)\Psi_{r}^{i}(z)dz}\\
&=\abs{\frac{1}{r}\int_{\zR^{n}}\pii
f(x-z)-f(x)\pdd\Psi_{r}^{i}(z)dz}\\
&=\abs{\frac{1}{r}\int_{\zR^{n}}\pii
f(z)-f(x)\pdd\Psi_{r}^{i}(x-z) dz}\\
&\leq\frac{1}{r}\int\limits_{B(x,2r)}
\!\!\!|f(z)-f(x)||\Psi_{r}^{i}(x-z)|dz+\frac{1}{r}\int\limits_{B^c(x,2r)}
\!\!\!|f(z)-f(x)||\Psi_{r}^{i}(x-z)| dz\\
&= I + II.
\end{split}
\end{displaymath}
We shall bound $I$ using \eqref{acot.origen},
\begin{displaymath}
\begin{split}
I&=\frac{1}{r}\int_{B(x,2r)}|f(z)-f(x)||\Psi_{r}^{i}(x-z)|dz\\
&\leq\frac{C}{r^{n+2}} \int_{B(x,2r)}|f(z)-f(x)||x-z| dz\\
&=\frac{C}{r^{n+2}}\sum_{j=0}^{\infty}\int_{\set{z:\, 2^{-j-1}\leq{\frac{|x-z|}{2r}<2^{-j}}}}|f(z)-f(x)||x-z| dz\\
&\leq\frac{C}{r^{n+2}}\sum_{j=0}^{\infty}\int_{B(x,2^{-j+1}r)}|f(z)-f(x)|2^{-j+1}r dz\\
&=\frac{C}{r^{n+1}}\sum_{j=0}^{\infty}2^{-j+1}\pii2^{-j+1}r \pdd^{n+\lambda}\frac{1}{\pii 2^{-j+1}r
\pdd^{n+\lambda}}\int_{B(x,2^{-j+1}r)}|f(z)-f(x)| dz\\
&\leq C r^{\lambda-1}\sum_{j=0}^{\infty}\pii 2^{-j+1}\pdd^{n+\lambda+1} M^{\sharp,\lambda}f(x)\\
&=C r^{\lambda-1}M^{\sharp,\lambda}f(x).
\end{split}
\end{displaymath}
Now from \textit{(f)} in Proposition~\ref{propo:propertiesPhi},
\begin{displaymath}
\begin{split}
II&=\frac{1}{r}\int_{B^{c}(x,2r)}|f(z)-f(x)||\Psi_{r}^{i}(x-z)|dz\\
&\leq\frac{C}{r}\sum_{j=0}^{\infty}\int_{\set{z:\,2^{j}\leq\frac{|x-z|}{2r}<2^{j+1}}}|f(z)-f(x)|
\frac{r^{2-a}}{|x-z|^{n+2-a}} dz\\
&\leq Cr^{1-a}\sum_{j=0}^{\infty}\int_{\set{z:\,2^{j}\leq\frac{|x-z|}{2r}<2^{j+1}}}
|f(z)-f(x)| \frac{1}{(2^{j+1}r)^{n+2-a}} dz\\
&\leq \frac{C}{r^{n+1}}\sum_{j=0}^{\infty}\pii
2^{j+1}\pdd^{-n-2+a}\frac{(r2^{j+2})^{n+\lambda}}{(r2^{j+2})^{n+\lambda}}\int_{B(x,2^{j+2}r)}|f(z)-f(x)|dz\\
&\leq Cr^{\lambda-1}\pii\sum_{j=0}^{\infty}\pii 2^{j+2}\pdd^{\lambda-2+a}\pdd M^{\sharp,\lambda}f(x)\\
&=C r^{\lambda-1}M^{\sharp,\lambda}f(x)
\end{split}
\end{displaymath}
and the Lemma is proved.
\end{proof}

\begin{proof}[Proof of Theorem~\ref{teo:improvementBesov}] Follows closely the lines of the
proof of Theorem~3 in \cite{DaDeV97}. The only point in which the
nonlocal character of our situation becomes relevant is contained
in the first estimates on page~11 in \cite{DaDeV97}. On the other hand, our upper restriction on
$\lambda$ is only a consequence of the fact that we are using only estimates for the first order derivatives
(after a fine tuning of the function $\varphi$ larger values of $\lambda$ can be achieved).
Our restriction guarantees the convergence of the series involved in the above mentioned estimates in \cite{DaDeV97}.
\end{proof}



\begin{thebibliography}{10}

\bibitem{AGconstructive}
Hugo Aimar and Ivana G{\'o}mez, \emph{Parabolic {B}esov regularity for the heat
  equation}, Constr. Approx. \textbf{36} (2012), no.~1, 145--159. \MR{2926308}

\bibitem{CaSi2007}
Luis Caffarelli and Luis Silvestre, \emph{An extension problem related to the
  fractional {L}aplacian}, Comm. Partial Differential Equations \textbf{32}
  (2007), no.~7-9, 1245--1260. \MR{2354493 (2009k:35096)}

\bibitem{DaDeV97}
Stephan Dahlke and Ronald~A. DeVore, \emph{Besov regularity for elliptic
  boundary value problems}, Comm. Partial Differential Equations \textbf{22}
  (1997), no.~1-2, 1--16. \MR{97k:35047}

\bibitem{DeSa84}
Ronald~A. DeVore and Robert~C. Sharpley, \emph{Maximal functions measuring
  smoothness}, Mem. Amer. Math. Soc. \textbf{47} (1984), no.~293, viii+115.
  \MR{85g:46039}

\bibitem{FaGa87}
Eugene~B. Fabes and Nicola Garofalo, \emph{Mean value properties of solutions
  to parabolic equations with variable coefficients}, J. Math. Anal. Appl.
  \textbf{121} (1987), no.~2, 305--316. \MR{872228 (88b:35088)}

\bibitem{FaKeSe82}
Eugene~B. Fabes, Carlos~E. Kenig, and Raul~P. Serapioni, \emph{The local
  regularity of solutions of degenerate elliptic equations}, Comm. Partial
  Differential Equations \textbf{7} (1982), no.~1, 77--116. \MR{84i:35070}

\bibitem{JeKe95}
David Jerison and Carlos~E. Kenig, \emph{The inhomogeneous {D}irichlet
  pro\-blem in {L}ipschitz domains}, J. Funct. Anal. \textbf{130} (1995),
  no.~1, 161--219. \MR{96b:35042}

\bibitem{Landkofbook}
N.~S. Landkof, \emph{Foundations of modern potential theory}, Springer-Verlag,
  New York, 1972, Translated from the Russian by A. P. Doohovskoy, Die
  Grundlehren der mathematischen Wissenschaften, Band 180. \MR{0350027 (50
  \#2520)}

\bibitem{Meyer92waveletsI}
Yves Meyer, \emph{Wavelets and operators}, Cambridge Studies in Advanced
  Mathematics, vol.~37, Cambridge University Press, Cambridge, 1992, Translated
  from the 1990 French original by D. H. Salinger. \MR{1228209 (94f:42001)}

\bibitem{NOS}
Ricardo~H. Nochetto, Enrique Ot\'{a}rola, and Abner~J. Salgado, \emph{A {PDE}
  approach to fractional diffusion in general domains: a priori error
  analysis}, Available in
  \href{http://arxiv.org/abs/1302.0698}{http://arxiv.org/abs/1302.0698}, 2013.

\bibitem{Silvestretesis}
Luis Silvestre, \emph{Regularity of the obstacle problem for a fractional power
  of the laplace operator}, Ph.D. thesis, The University of Texas at Austin,
  2005.

\bibitem{Stein70}
Elias~M. Stein, \emph{Singular integrals and differentiability properties of
  functions}, Princeton Mathematical Series, No. 30, Princeton University
  Press, Princeton, N.J., 1970. \MR{0290095 (44 \#7280)}

\bibitem{ToSti2010}
Pablo~Ra{\'u}l Stinga and Jos{\'e}~Luis Torrea, \emph{Extension problem and
  {H}arnack's inequality for some fractional operators}, Comm. Partial
  Differential Equations \textbf{35} (2010), no.~11, 2092--2122. \MR{2754080
  (2012c:35456)}

\end{thebibliography}

\def\cprime{$'$}
\providecommand{\bysame}{\leavevmode\hbox to3em{\hrulefill}\thinspace}
\providecommand{\MR}{\relax\ifhmode\unskip\space\fi MR }
\providecommand{\MRhref}[2]{%
  \href{http://www.ams.org/mathscinet-getitem?mr=#1}{#2}
}
\providecommand{\href}[2]{#2}


\bigskip
\bigskip

{\footnotesize
\textsc{Instituto de Matem\'atica Aplicada del Litoral (IMAL), CONICET-UNL}


\smallskip
\textmd{G\"{u}emes 3450, S3000GLN Santa Fe, Argentina.}
}

\bigskip

\end{document}